\documentclass[a4paper, 10pt]{article}
\usepackage{verbatim}
\usepackage{hyperref}
\usepackage{amsmath,amssymb}
\newtheorem{theorem}{Theorem}

\newtheorem{definition}[theorem]{Definition}

\newtheorem{lemma}[theorem]{Lemma}

\newtheorem{question}[theorem]{Question}

\newtheorem{proposition}[theorem]{Proposition}

\newenvironment{proof}[1][Proof]{\noindent\textbf{#1.} }{\ \rule{0.5em}{0.5em}}
\topmargin by -\headsep
\title{An analytical Lieb-Sokal lemma}
\begin{document}

\author{Mohan Ravichandran\footnote{mohan.ravichandran@msgsu.edu.tr, Supported by Tubitak 1001 grant number 115F204}}
\date{}
\maketitle

\begin{abstract}
A polynomial $p \in \mathbb{R}[z_1, \cdots, z_n]$ is called real stable if it is non-vanishing whenever all the variables take values in the upper half plane. A well known result of Elliott Lieb and Alan Sokal states that if $p$ and $q$ are $n$ variate real stable polynomials, then the polynomial $q(\partial)p := q(\partial_1, \cdots, \partial_n)p$, is real stable as well. In this paper, we prove analytical estimates on the locations on the zeroes of the real stable polynomial $q(\partial)p$ in the case when both $p$ and $q$ are multiaffine, an important special case, owing to connections to negative dependance in discrete probability. As an application, we prove a general estimate on the expected characteristic polynomials upon sampling from Strongly Rayleigh distributions. We then use this to deduce results concerning two classes of polynomials, mixed characteristic polynomials and mixed determinantal polynomials, that are related to the Kadison-Singer problem.  
\end{abstract}

\section{Introduction}

The study of real rooted univariate polynomials and linear transformations that preserve real rootedness has a long history going back to Schur and Polya and with important contributions by Sz. Nagy and Szego, among others \cite{RS02}[Ch. 5]. Given a univariate real rooted polynomial $p$, it is easy to see that the polynomial $(D-\alpha)p$, where $D$ denotes the derivative operator, is also real rooted. Iterating this, we see that if $q$ is another real rooted polynomial, then the polynomial $q(D)p$ is also real rooted\footnote{We will call constant polynomials real rooted, to keep the statements in this paper concise}.

Borcea and Branden in a celebrated series of papers \cite{BBLY1, BBLY2, BBPS, BBJ} developed a powerful theory of \emph{real stability}, a multivariate generalization of real rootedness. A polynomial $p(z_1, \cdots, z_n) \in \mathbb{R}[z_1, \cdots, z_n]$ is called \emph{real stable} if, 
\[p(z_1 ,\cdots, z_n) \neq 0, \quad \text{ if } \quad \operatorname{Im}(z_i) > 0, \,\,\forall i \in [n].\]

Borcea and Branden proved a variety of results concerning linear operators that preserve real stability; The result that will concern us here is an earlier result due to Lieb and Sokal \cite{LiSo}\footnote{This is a special case of a more general result due to Lieb and Sokal, which says that replacing a variable with a partial derivative with respect to a second variable preserves real stability.}, that is a direct analogue of the result mentioned earlier concerning univariate real stability preservers. Given real stable $q(z_1, \cdots, z_n)$ and $p(z_1, \cdots, z_n)$, then the polynomial, 
\[r(z_1, \cdots, z_n) := [q(\partial_1, \cdots, \partial_n)p](z_1, \cdots, z_n),\]
is also real stable. 

This is a \emph{qualitative result}; A celebrated recent application of real stability, the solution of the Kadison-Singer problem due to Marcus, Spielman and Srivastava \cite{MSS2}, required a \emph{quantitative} analogue of this result for an important class of real stable polynomials called \emph{mixed characteristic polynomials}. Given PSD matrices $A_1, \cdots, A_m$, Marcus, Spielman and Srivatava consider the following polynomial, 
\[\mu[A_1, \cdots, A_m] (x):= \left(\prod_{i = 1}^{m} (1-\partial_i)\right) \operatorname{det}[xI + z_1 A_1 + \cdots + z_m A_m]\mid_{z_1 = \cdots = z_m = 0}.\]

Results of Borcea and Branden \cite{BBLY2}, show that this polynomial is real rooted. But the crux of MSS' solution \cite{MSS2} to the Kadison-Singer problem lies in proving estimates on where the roots of the polynomial lie. MSS showed, using their \emph{multivariate barrier method} that if $\operatorname{Trace}(A_k) \leq \epsilon$ for $1 \leq k \leq m$ and if $A_1 + \cdots + A_m = I_n$, then the largest root of $\mu[A_1, \cdots, A_m]$ is at most $(1+\sqrt{\epsilon})^2$. 

MSS \cite{MSS2}, deduce this from a more general result of theirs'.  Under the same hypotheses, the following polynomial which is real stable, 
\[ \left(\prod_{i = 1}^{m} (1-\partial_i)\right) \operatorname{det}[xI + z_1 A_1 + \cdots + z_m A_m],\]
is non-zero whenever all the $z_k$ are real and $z_k  >  (1+\sqrt{\epsilon})^2$ for all $k \in [n]$. 

Another class of polynomials related to the Kadison-Singer problem are the \emph{mixed determinantal polynomials} introduced in \cite{MRKSP}. We discuss here, a special case. Given a hermitian matrix $A \in M_n(\mathbb{C})$, consider the polynomial, 
\[\chi_2[A](x) := \partial^{[n]}\operatorname{det}[Z-A]^2\mid_{z_1 = \cdots = z_n = x} := \dfrac{\partial^{n}}{\partial z_1\, \cdots \, \partial z_n}\operatorname{det}[Z-A]^2\mid_{z_1 = \cdots = z_n = x}.\]

In \cite{MRKSP}, the author gave a direct proof of Anderson's paving conjecture, another statement that implies the Kadison-Singer problem by proving that if $A$ is a hermitian contraction and the diagonal entries of $A$ are all at most $\alpha < -1/2$, then the largest root of $\chi_2[A]$ is strictly less than $1$. Akin to the work of MSS, this was deduced from a result that states that under the same hypotheses, the polynomial $\partial^{[n]}\operatorname{det}[Z-A]^2$ is non-vanishing whenever all the $z_i$ are real and at most $\dfrac{\sqrt{3(1-\alpha^2)}-\alpha}{2}$.

 These two results lead to the following natural question, 
\begin{question}
Given real stable polynomials, $q(z_1, \cdots, z_n)$ and $p(z_1, \cdots, z_n)$, get estimates on the locations of zero free regions in $\mathbb{R}^n$ of the real stable polynomial, 
\[r(z_1, \cdots, z_n) := [q(\partial_1, \cdots, \partial_n)p](z_1, \cdots, z_n).\]
\end{question}

We take the liberty of calling such results \emph{Analytical Lieb-Sokal} lemmas. In this paper, we prove the first such general result, in the case when both $p$ and $q$ are, additionally, \emph{multiaffine}. This might seem overly restrictive, but we now give three reasons to point out why this is nevertheless, interesting. 

Firstly, Borcea and Branden's celebrated results on real stability employ a reduction to the multiaffine case, through the operation they call \emph{polarization}. Given a real stable polynomial $p(z_1, \cdots, z_n)$ of degree $k_i$ in the variable $z_i$, the multaffine polynomial gotten by introducing new variables $\{z_{i,j}, \, j \in [k_i], \in i \in [n]\}$ and replacing each occurrence of $z_{i}^k$ by $\dfrac{e_{k}(z_{i,1}, \cdots, z_{i,k_i})}{\binom{k_i}{k}}$, is called the polarization of $p$. Borcea and Branden \cite{BBLY1}, prove, using the Grace-Walsh-Szego theorem, that the polarization operation preserves real stability. Polarization also allows us to transfer several questions about general real stable polynomials to questions about multiaffine real stable polynomials. 

Secondly, one of the most exciting applications of real stable polynomials has been to discrete probability. Borcea, Branden and Liggett \cite{BBL09}, used real stability to introduce the class of \emph{Strongly Rayleigh} measures, a large class of measures that enjoy excellent negative dependance properties. A probability measure $\mu$ on $2^{[n]}$ is said to be Strongly Rayleigh if the associated generating polynomial (which is multiaffine), 
\[P_{\mu}(z_1, \cdots, z_n) := \sum_{S \subset [n]} \mu(S)z^{S},\]
is real stable.  Product measures are Strongly Rayleigh, as are the important class of determinantal measures. Strongly Rayleigh measures are invariant under a number of operations like conditioning and the addition of external fields and this marks them out from other notions of negative correlation. They have been shown recently to satisfy a natural central limit theorem \cite{GLP16}, that further indicates both their tractability and their relevance. 

Thirdly, both classes of polynomials of direct relevance to the Kadison-Singer problem, namely mixed characteristic polynomials and mixed determinantal polynomials, have alternate expressions as a polynomial of the form $q(\partial) p$ where $q$ and $p$ are multiaffine. Our analytical Lieb-Sokal lemma will also allow us to prove asymptotic results about the roots of mixed determinantal polynomials, something that the results in \cite{MRKSP} did not yield. 

The paper is organised as follows: In the next section, we give definitions as well as an extended motivation for the main result. Section $3$ consists of combinatorial lemmas that will be needed to prove the main result. The proof of the analytical Lieb-Sokal lemma is given in section $4$. In section $5$, we apply this lemma to four classes of polynomials to showcase its utility. We end with a final section on possible extensions and other applications. 

\section{Preliminaries}

Given a univariate real rooted polynomial $p$, it is easy to check that the polynomial $(D - \alpha)p$ is also real rooted.\footnote{Here, $D$ denotes the derivative.} Further, it is elementary to see that, 
\[\lambda_{max} [(D-\alpha)p] \geq \lambda_{max} [p], \,\,\text{ if } \, \alpha > 0, \quad \lambda_{max} [(D-\alpha)p] \leq \lambda_{max}[p]\,\,\text{ if } \,\,\alpha \leq 0.\]

The maximum root of a polynomial is highly sensitive to differential transformations and there is little one can say beyond the above concerning how much the maximum root shifts. However, a remarkable idea due to Marcus, Spielman and Srivastava \cite{MSSICM}, allows one to get useful estimates on  root locations upon iteratively applying linear differential operators. Given a real rooted polynomial $p$, MSS consider the following quantity, 
\[\operatorname{smax}_{\varphi}(p) := \operatorname{max\,root}[Dp-\varphi\, p],\]
where $\varphi \geq 0$. One checks that $\operatorname{smax}_{\varphi}(p) \geq \lambda_{max}(p)$ with equality exactly when $\varphi = \infty$ and also, that  $\operatorname{smax}_{0}(p) = \infty$. Significantly, $\operatorname{smax}_{\varphi}$ for bounded values of $\varphi$, behaves well upon taking differential operators. One has the following key result, due to Adam Marcus,
\begin{proposition}[Marcus, 2014]
Let $p$ be a real rooted polynomial and let $\varphi \geq 0$. Then, for any $\alpha \leq 0$, we have, 
\[\operatorname{smax}_{\varphi}[(D-\alpha)p] \leq \operatorname{smax}_{\varphi}[p] - \dfrac{1}{\varphi - \alpha}.  \]
\end{proposition}

Iterating this, we have, 
\begin{proposition}
Let $p$ be a real rooted polynomial and let $q$ be a real rooted polynomial with all roots non-positive and let $\varphi \geq 0$. Then, 
\[\operatorname{smax}_{\varphi}[q(D)p] \leq \operatorname{smax}_{\varphi}[p] - \sum_{\alpha \in \sigma(q)}\dfrac{1}{\varphi - \alpha}.  \]
\end{proposition}
This proposition can often be used to get useful estimates by optimizing over $\varphi$. When $q = z^k$, one derives asymptotically tight estimates in the restricted invertibility problem \cite{MRI}. Our goal in this paper is a multivariate analogue of the above result.

Recall that a polynomial $p(z_1, \cdots, z_n) \in \mathbb{C}[z_1, \cdots, z_n]$ is called \emph{stable} if,
\[p(z_1, \cdots, z_n) \neq 0, \quad \text{ whenever } \quad Im(z_i) > 0, \, \, \forall i \in [n].\]

If the coefficients are all, additionally, real, then it is called \emph{real stable}. It is elementary to check that univariate real rooted polynomials with real coefficients are real stable and conversely, that univariate real stable polynomials are real rooted. Several natural operations preserve real stabililty, some of them, being, 
\begin{enumerate}
\item Partial differentiation: If $p(z_1, \cdots, z_n)$ is real stable, then so is $\partial_1 p$. 
\item Specialization: $p(a, z_2, \cdots, z_n)$ is real stable for any $a \in \mathbb{R}$.
\item Pair Interactions: $p(\partial_1, z_2, \cdots, z_n)$ is real stable. By this we mean the following: Suppose 
\[p = \sum_{k \geq 1, S\subset [2, \cdots, n]} a_{k, S}\,z_1^k z^{S},\]
is real stable, then, so is, 
\[p = \sum_{k \geq 1, S\subset [2, \cdots, n]}
 a_{k, S} \,\partial_2^k\, (z^{S}).\] This is the well known Lieb-Sokal lemma \cite{LiSo}[Prop. 2.2].
\end{enumerate}

Simply iterating these three operations allows us to generate a large class of real stable polynomials. 

In \cite{MSSPC}, Marcus, Spielman and Srivastava studied the following operation on real rooted polynomials, which they call the \emph{symmetric additive convolution}. Given real rooted monic polynomials $p$ and $q$ of degree $d$, define, 
\[(p \boxplus_d q)(x+y)  := \dfrac{1}{d!}\sum_{k = 0}^{d} p^{(k)}(x)q^{(d-k)}(y)\tag{Free additive convolution}.\]
MSS showed that this operation is well defined and preserves real-rootedness and further, proved tight estimates on the location of its' roots, an estimate that was used by them in their proof of the existence of bipartite regular Ramanujan multigraphs of every degree and number of vertices.

\begin{theorem}[MSS, 2015]\label{roots}
	Let $p$ be monic, real rooted and degree $d$. Then, 
	\[\operatorname{smax}_{\varphi} (p \boxplus_d q) \leq \operatorname{smax}_{\varphi} (p) + \operatorname{smax}_{\varphi}(q) - \dfrac{d}{\varphi}.\]
	\end{theorem}
 
Adam Marcus, in a follow up paper \cite{MarCon}, studied this operation in further detail and made a compelling case for seeing this operation as a finite dimensional avatar of Voiculescu's \emph{free additive convolution} \cite{DNV}. We now show how this \emph{finite free additive convolution} has an alternate expression as the specialization of a multivariate polynomial, an expression that will naturally lead to generalizations. Let us make the following definition,
	\begin{definition}\label{def2}
	Let $p(z_1, \cdots, z_n)$ and $q(z_1, \cdots, z_n)$ be multiaffine real stable polynomials. Define,
	\[[p\ast q](z_1 + y_1, \cdots, z_n + y_n) := \left[\prod_{i = 1}^{n} \left(\dfrac{\partial}{\partial z_i} + \dfrac{\partial}{\partial y_i}\right)\right]p(z_1, \cdots, z_n)q(y_1, \cdots, y_n).\]
	\end{definition}
	It is easy to show that this is well defined. One also has, 
	\[[p\ast q](2z_1, \cdots, 2z_n) =  \left[\dfrac{\partial^{n}}{\partial z_1 \cdots \partial z_n} (pq) \right](z_1, \cdots, z_n).\]

	MSS' root estimate (\ref{roots}) can be neatly expressed using the convolution $p \ast q$ defined on multiaffine real stable polynomials. Let us use the notation $\operatorname{Pol}(p)$ to denote the polarization of a univariate polynomial $p$.\footnote{If $p = \sum_{k = 0}^{n} a_k z^{k}$, then $Pol(p) = \sum_{k = 0}^{n} a_k\dfrac{e_{k}(z_1, \cdots, z_n)}{\binom{n}{k}}$.} One has the following fact, 
	\begin{lemma}\label{convol}
		Let $p$ and $q$ be monic, real rooted and degree $d$. Then, 

	\[(p \boxplus_d q)(x) = [\operatorname{Pol}(p) \ast \operatorname{Pol}(q) ](x, \cdots, x).\]
	\end{lemma}
	
	Theorem (\ref{roots}) using the notation of lemma (\ref{convol}) becomes,
		\begin{proposition}[Marcus, Spielman, Srivastava]\label{derr}
	Let $p(z_1, \cdots, z_n)$ and $q(z_1, \cdots, z_n)$ be multiaffine real stable polynomials and further assume that they are symmetric(invariant under permutations of the variables). Let $\textbf{a}$ and $\textbf{b}$ be above the roots of $p$ and $q$ respectively(in the notation of MSS). Assume that, 
	\[\dfrac{\partial_i p}{p}(\textbf{a}) \leq \varphi_1, \quad \dfrac{\partial_i q}{q}(\textbf{b}) \leq \varphi_2, \quad i \in [n].\]
Then, letting $\textbf{c} = \dfrac{\textbf{a} + \textbf{b}}{2} - \dfrac{2}{\varphi_1 + \varphi_2}\textbf{1}$, we have that $\textbf{c}$ is above the roots of $p \ast q$ and further, 
\[\dfrac{\partial_i (p\ast q)}{p \ast q}(\textbf{c}) \leq \dfrac{\varphi_1 + \varphi_2}{2}.\] 
	\end{proposition}
	
	It is natural to ask if something similar holds when the assumption of symmetry is dropped. It does not, but we will prove a weaker result, 
	
	\begin{proposition}
	Let $p(z_1, \cdots, z_n)$ and $q(z_1, \cdots, z_n)$ be multiaffine real stable polynomials and let $\textbf{a}$ and $\textbf{b}$ be above the roots of $p$ and $q$ respectively. Assume that, 
	\[\dfrac{\partial_i p}{p}(\textbf{a}) \leq \varphi_1, \quad \dfrac{\partial_i q}{q}(\textbf{b}) \leq \varphi_2, \quad i \in [n].\]
Then, letting $\textbf{c} = \dfrac{\textbf{a} + \textbf{b}}{2} - \dfrac{1}{\varphi_1 + \varphi_2}\textbf{1}$, we have that $\textbf{c}$ is above the roots of $p \ast q$ and further, 
\[\dfrac{\partial_i (p\ast q)}{p \ast q}(\textbf{c}) \leq \dfrac{\varphi_1 + \varphi_2}{2}.\] 
	\end{proposition}
	
	This is different from (\ref{derr}) because of the term $\dfrac{1}{\varphi_1 + \varphi_2}$ instead of $\dfrac{2}{\varphi_1 + \varphi_2}$, rendering the estimate a little weaker.

Going back to definition (\ref{def2}), let us recall that $p \ast q$ can be written as a differential operator with a multiaffine symbol applied to another multiaffine polynomial. One has, 

	\[[p\ast q](z_1 + y_1, \cdots, z_n + y_n) := \left[\prod_{i = 1}^{n} \left(\dfrac{\partial}{\partial z_i} + \dfrac{\partial}{\partial y_i}\right)\right]p(z_1, \cdots, z_n)q(y_1, \cdots, y_n).\]

We are now naturally led to ask a more general question.

\begin{question}\label{gen}
Let $p$ and $q$ be multiaffine real stable polynomials in $\mathbb{R}[z_1, \cdots, z_n]$. Derive estimates for the location of the zero free regions of the real stable polynomial, 
\[r := q(\partial)p := q(\partial_1 ,\cdots, \partial_n)p.\]
\end{question}

We will be able to give a usable answer to this question in section (4) of this paper. We will then indicate four applications to this result to combinatorial polynomials related to the Kadison-Singer problem.

\section{Convolutions of multiaffine real stable polynomials}

We start off with defining a natural convolution operation on multiaffine real stable polynmials.
\begin{definition}
Given two multiaffine real stable polynomials $p, q$ in $\mathbb{R}[z_1, \cdots, z_n]$ define their convolution $p \ast q$ as, 
\[(p\ast q)(2z_1, \cdots, 2z_n) := \left[\dfrac{\partial^n}{\partial z_1 \, \cdots \, z_n} (qp)\right](z_1, \cdots, z_n).\]
\end{definition}

Since the product of real stable polynomials is real stable and taking partial derivatives preserves real stability, it is immediate that $p \ast q$ is also real stable. It is also easy to see that $p \ast q$ is additionally, multiaffine. We now collect a few elementary facts about the above convolution. The first fact is an alternate expression that justifies our usage of the term convolution to describe it.
\begin{lemma}\label{tl1}
Let $p, q$ be multiaffine real stable polynomials in $\mathbb{R}[z_1, \cdots, z_n]$. Let us write them out as, 
\[p(z_1, \cdots, z_n) = \sum_{S \subset [n]} p_S z^{S}, \qquad q(z_1, \cdots, z_n) = \sum_{S \subset [n]} q_S z^{S}.\]
Then, we have that, 
\[(p \ast q)(z_1, \cdots, z_n) = \sum_{S \subset [n]}\left(\sum_{\substack{ R \cup T = [n],\\  R \cap T = S}} p_R\, q_T\right) z^{S}.\]
\end{lemma}
\begin{proof}
By linearity, it is enough to prove this when $p = z^{S_1}$ and $q = z^{S_2}$. We then have that, 
\[(p\ast q)(2z_1, \cdots, z_n) = \dfrac{\partial^n}{\partial z_1 \, \cdots \, z_n} z^{S_1}z^{S_2}.\]
This is zero unless $S_1 \cup S_2 = [n]$, in which case, we have that $z^{S_1}z^{S_2} = z^{[n]}z^{S_1 \cap S_2}$, yielding that, 
\[(p\ast q)(2z_1, \cdots, z_n) = \dfrac{\partial^n}{\partial z_1 \, \cdots \, z_n} z^{[n]}z^{S_1 \cap S_2} = 2^{|S_1 \cap S_2|}z^{S_1 \cap S_2}.\]
We conclude that if $S_1 \cup S_2 = [n]$, 
\[z^{S_1} \ast z^{S_2} = z^{S_1 \cap S_2}.\]
By linearity, this completes the proof. 
\end{proof}

The second elementary fact shows that this convolution can be written as a differential operator with multiaffine symbol applied to another multiaffine real stable polynomial.

\begin{lemma}\label{tl2}
Let $p$ and $q$ be real stable multiaffine polynomials in $\mathbb{R}[z_1, \cdots, z_n]$. Then, 
\[ (p \ast q)(z_1+y_1, \cdots, z_n+y_n) = \left[\prod_{i = 1}^{n} \left(\dfrac{\partial}{\partial z_i} + \dfrac{\partial}{\partial y_i}\right) \right]  p(z_1, \cdots, z_n) q(y_1, \cdots, y_n).\]
\end{lemma}
\begin{proof}
We need to show that, 
\[\left[\prod_{i = 1}^{n} \left(\dfrac{\partial}{\partial z_i} + \dfrac{\partial}{\partial y_i}\right) \right]  p(z_1, \cdots, z_n) q(y_1, \cdots, y_n) = \left[\dfrac{\partial^n}{\partial z_1 \, \cdots \, z_n} pq\right]\left(\frac{z_1+y_1}{2}, \cdots, \frac{z_n+y_n}{2}\right).\]
By bilinearity, it is enough to show this for monomials. Let $p = z^{S}$ and $q = z^{T}$. Both expressions above are zero unless $T \cup S = [n]$. For the non-trivial case, let us write $S = A \cup C$ and $T = B \cup C$ where $A, B, C$ form a partition of $[n]$. For the right hand expression, we have, 
\[ \dfrac{\partial^n}{\partial z_1 \, \cdots \, z_n} z^{A}z^{B}z^{2C} = 2^{|C|}z^{C}.\] Evaluating this at the point $\left(\dfrac{z_1+y_1}{2}, \cdots, \dfrac{z_n+y_n}{2}\right)$, we get the expression,
\begin{eqnarray}\label{one}
 \left[\dfrac{\partial^n}{\partial z_1 \, \cdots \, z_n} \right] z^{A}z^{B}z^{2C} \left(\frac{z_1+y_1}{2}, \cdots, \frac{z_n+y_n}{2}\right) = \prod_{i \in S} (z_i + y_i).
 \end{eqnarray}

For the left hand expression, we have that, 
\[\left[\prod_{i = 1}^{n} \left(\dfrac{\partial}{\partial z_i} + \dfrac{\partial}{\partial y_i}\right) \right] = \sum_{S \subset [n]} \dfrac{\partial^{S}}{\partial z^{S}} \dfrac{\partial^{S^{c}}}{\partial y^{S^{c}}}.\]
We also note that, 
\[\dfrac{\partial^{S}}{\partial z^{S}} \dfrac{\partial^{S^{c}}}{\partial y^{S^{c}}} z^{A}z^{C}y^{B}y^{C},\]
is zero unless $S \subset A \cup B$ and $S^{c} \subset B \cup C$, which means we can write $S = A \cup D$, where $D \subset C$, in which case, we get, 
\[\dfrac{\partial^{S}}{\partial z^{S}} \dfrac{\partial^{S^{c}}}{\partial y^{S^{c}}} z^{A}z^{C}y^{B}y^{C} = z^{C \setminus D}y^{D}.\]
Summing this up, we have, 
\begin{eqnarray}\label{two}
\nonumber \left[\prod_{i = 1}^{n} \left(\dfrac{\partial}{\partial z_i} + \dfrac{\partial}{\partial y_i}\right) \right] z^{A \cup C}y^{B \cup C} &=& \left[\sum_{S \subset [n]} \dfrac{\partial^{S}}{\partial z^{S}} \dfrac{\partial^{S^{c}}}{\partial y^{S^{c}}}\right] z^{A \cup C}y^{B \cup C} \\
\nonumber &=& \sum_{D \subset C} z^{C \setminus D}y^{D}\\
&=& \prod_{i \in S} (z_i + y_i)
\end{eqnarray}
Comparing (\ref{one}) and (\ref{two}), we see that the lemma holds.
\end{proof}

We now relate our convolution of multiaffine polynomials to the result of applying a multiaffine differential operator to another multiaffine polynomial.
\begin{lemma}\label{tl3}
Let $p$ and $q$ be multiaffine real stable polynomials in $\mathbb{R}[z_1, \cdots, z_n]$ and let $r = q(\delta)p$. Then,
\[r(2z_1, \cdots, 2z_n) = \left[\dfrac{\partial^n}{\partial z_1 \, \cdots \, z_n} \overline{q}p\right](z_1, \cdots, z_n),\]
where $\overline{q}$ is the flip of $q$, that is, if we write 
\[q = \sum_{S} a_S z^{S},\]
then,
\[\overline{q}:=\sum_{S}a_{S}z^{[n]\setminus S}.\]
\end{lemma}
\begin{proof}
By linearity, it is enough to check this when $p$ and $q$ are monomials. Let $p = z^{S}$ and $q = z^{T}$. Then, $r := q(\delta)p = z^{S \setminus T}$ if $S \supset T$ and is zero otherwise. We have that, 
\[ \left[\dfrac{\partial^n}{\partial z_1 \, \cdots \, z_n} \overline{q}p\right](z_1, \cdots, z_n) = \left[\dfrac{\partial^n}{\partial z_1 \, \cdots \, z_n} z^{[n]\setminus T}z^{S}\right](z_1, \cdots, z_n). \]
This is zero if $T$ is not a subset of $S$ and otherwise, letting $S = T \cup R$, we have, 
\[\left[\dfrac{\partial^n}{\partial z_1 \, \cdots \, z_n} \overline{q}p\right] = \left[\dfrac{\partial^n}{\partial z_1 \, \cdots \, z_n} z^{([n]\setminus S) \cup R}z^{S}\right] = \left[\dfrac{\partial^n}{\partial z_1 \, \cdots \, z_n} z^{[n] \cup R}\right] = 2^{|R|}z^{R} .\]
The assertion follows. 
\end{proof}

Combining lemmas (\ref{tl1}) and (\ref{tl3}),  we get, 
\begin{proposition}
Let $p$ and $q$ be multiaffine real stable polynomials in $\mathbb{R}[z_1, \cdots, z_n]$ and let $r = q(\delta)p$. Then,
\[ r(z_1+y_1, \cdots, z_n+y_n) = \left[\prod_{i = 1}^{n} \left(\dfrac{\partial}{\partial z_i} + \dfrac{\partial}{\partial y_i}\right) \right]  p(z_1, \cdots, z_n) \overline{q}(y_1, \cdots, y_n).\]
\end{proposition}

Notice that this expression involves a differential operator with multiaffine symbol applied to another multiaffine polynomial. In the next section, we will exploit this to get non-trivial root shift estimates. 

\section{An analytical Lieb-Sokal lemma}

In this section, we will prove the following proposition for getting root estimates,

	\begin{theorem}[An analytical Lieb-Sokal lemma]\label{ALS}
Let $p$ and $q$ be multiaffine real stable polynomials in $\mathbb{R}[z_1, \cdots, z_n]$ and let $r = q(\partial)p$. Let $\textbf{a}$ lie above the roots of $p$ and $\textbf{b}$ lie above the roots of $\overline{q}$. Then, 
\[\textbf{c} := \textbf{a}+\textbf{b} - \left(\dfrac{1}{\varphi_1},\cdots, \dfrac{1}{\varphi_n}\right),\]
lies above the roots of $q(\partial)p$ where, 
\[\varphi_i = \dfrac{\partial_i \,p}{p}(\textbf{a}) + \dfrac{\partial_i \, \overline{q}}{\overline{q}}(\textbf{b}), \qquad i \in [n].\] 
\end{theorem}

 Recall the following definitions due to Marcus, Spielman and Srivastava \cite{MSS2}. Given a real stable polynomial $p(z_1, \cdots, z_n)$, a point $\textbf{a}$ is said to \emph{above the roots} of $p$ if,
\[p(\textbf{a} + \textbf{t}) \neq 0, \quad \forall \textbf{t} \in \mathbb{R}^n_{+}.\]
 Also, the logarithamic potential function in the direction $k$ for a real stable polynomial $p$ at a point $\textbf{a}$ above its roots is defined as, 
\[\Phi^{k}_{p}(\textbf{a}) := \dfrac{\partial_k p}{p}(\textbf{a}).\]
We will consider a slight generalization of this: Given a unit vector $v \in \mathbb{R}^n_{+}$, we will need the directional derivative of $\operatorname{log}(p)$ in the direction $v$, 
\[\Phi^{v}_{p}(\textbf{a}) := \dfrac{D_v p}{p}(\textbf{a}).\]
 We will require the following fact, which is a slight generalization of a result of MSS. 
\begin{lemma}\label{convex}
Let $p(z_1, \cdots, z_n)$ be real stable and let $\textbf{a}$ be above the roots of $p$. Then, for any two unit vectors $v, w \in \mathbb{R}^{n}_{+}$, we have that,
\begin{enumerate}
\item $\Phi^{v}_{p}(\textbf{a}) \geq 0$.
\item $D_{w}\left(\Phi^{v}_{p}\right)(\textbf{a}) \leq 0$.
\item $D^2_{w}\left(\Phi^{v}_{p}\right)(\textbf{a}) \geq 0$.
\end{enumerate}
\end{lemma}
\begin{proof}
Define a bivariate polynomial in the following fashion, 
\[q(z_1, z_2):= p(\textbf{a} + z_1 v + z_2 w).\]
We claim that $q$ is real stable. We need to check that if $\operatorname{Im}(z_1)$ and $\operatorname{Im}(z_2)$ are positive, then $q(z_1, z_2) \neq 0$. Noting that $v, w \in \mathbb{R}^n_{+}$, we see that, 
\[\operatorname{Im}(\textbf{a} + z_1 v + z_2 w) = \operatorname{Im}(z_1) v + \operatorname{Im}(z_2) w \in \mathbb{R}^n_{+}.\]
The desired conclusion follows from the real stability of $p$. Results from \cite{MSS2} show that, 
\[\Phi_q^{1}(0,0) \geq 0, \quad \partial_2 \Phi_q^{1}(0,0) \leq 0, \quad \partial_2^2 \Phi_q^{1}(0,0) \geq 0.\]

Next, it is clear given that $\textbf{a}$ is above the roots of $p$, that $(0, 0)$ is above the roots of $q$. We check that, 
\[\Phi_p^{v}(\textbf{a}) =\Phi_q^{1}(0,0),\quad  D_{w}\Phi_p^{v}(\textbf{a}) = \partial_2 \Phi_q^{1}(0, 0), \quad  D_{w}^{2}\Phi_p^{v}(\textbf{a}) = \partial_2^{2} \Phi_q^{1}(0, 0),\]
yielding the desired conclusions.
\end{proof}

En route to the Analytical Lieb-Sokal lemma, we prove a lemma on how potentials change upon taking certain derivatives.
\begin{lemma}\label{shift}
Let $p$ be a multiaffine real stable polynomials in $\mathbb{R}[z_1, \cdots, z_n]$ and $\textbf{a}$ be above the roots of $p$. Then, for any $i, j, k \in [n]$ we have that if,
\[\delta_i + \delta_j \leq \dfrac{p}{\partial_i p+ \partial_j p}(\textbf{a}),\]
then, 
\[\Phi^{k}_{(\partial_i + \partial_j)p}(\textbf{a}
 - \delta_i e_i - \delta_j e_j)   \leq \Phi_p^k(\textbf{a}).\]
\end{lemma}
\begin{proof}
Let us first assume that $i, j, k$ are all distinct. (The other cases are even easier). We may assume without loss of generality that $i=1$ and $j = 2$ and $k = 3$. Suppose that, 
\[\dfrac{\partial_3(\partial_1 + \partial_2)p}{(\partial_1 + \partial_2)p}(\textbf{a}
 - \delta_1 e_1 - \delta_2 e_2) \leq \dfrac{\partial_3 p}{p}(\textbf{a}).\]
Recall that $p$ is assumed to be multiaffine. We note that, 
\begin{eqnarray*}
(\partial_1 + \partial_2)p(\textbf{a} - \delta_1 e_1 - \delta_2 e_2) &=& (\partial_1 + \partial_2)\left[p(\textbf{a}) - \delta_1 \partial_1p(\textbf{a}) - \delta_2 \partial_2(\textbf{a}) + \delta_1 \delta_2 \partial_1 \partial_2p(\textbf{a})\right]\\
&=& (\partial_1 + \partial_2)p(\textbf{a}) - (\delta_1 + \delta_2)\partial_1 \partial_2 p(\textbf{a}).
\end{eqnarray*}
We need to find $\delta_1$ and $\delta_2$ such that, 
\[\dfrac{\partial_3(\partial_1 + \partial_2)p(\textbf{a}) - (\delta_1 + \delta_2)\partial_1 \partial_2 \partial_3 p(\textbf{a})}{(\partial_1 + \partial_2)p(\textbf{a}) - (\delta_1 + \delta_2)\partial_1 \partial_2 p(\textbf{a})} \leq \dfrac{\partial_3 p}{p}(\textbf{a}).\]
This simplifies to, 

\[\partial_3 \left(\dfrac{(\partial_1 + \partial_2)p}{p}\right)(\textbf{a}) \leq \dfrac{\delta_1 + \delta_2}{2} \partial_3\left(\dfrac{(\partial_1 + \partial_2)^2p}{p}\right)(\textbf{a}).\]
We may rewrite this as, 
\[\partial_3 \left(\dfrac{(\partial_1 + \partial_2)p}{p}\right)(\textbf{a}) \leq \dfrac{\delta_1 + \delta_2}{2} (\partial_1+\partial_2)^2 \left(\dfrac{\partial_3 p}{p}(\textbf{a}) \right) + (\delta_1 + \delta_2)\partial_3 \left(\dfrac{(\partial_1 + \partial_2)p}{p}\right)(\textbf{a})\left(\dfrac{(\partial_1 + \partial_2)p}{p}\right)(\textbf{a}) .\]
Using lemma (\ref{convex}), we observe that, 
 \[ (\partial_1+\partial_2)^2 \left(\dfrac{\partial_3 p}{p} \right)(\textbf{a}) \geq 0.\]
We conclude that if, 
\[\delta_1 + \delta_2 \leq \dfrac{1}{\partial_1 p+ \partial_2 p}(\textbf{a}),\]
then, 
\[\dfrac{\partial_3 (\partial_1 + \partial_2)p}{(\partial_1 + \partial_2)p}(\textbf{a}-\delta_1 e_1 - \delta_2 e_2) \leq \dfrac{\partial_3 p}{p}(\textbf{a}).\]

\end{proof}

Iterating the above argument, we have, 
\begin{lemma}\label{shift2}
Let $p$ be a multiaffine real stable polynomials in $\mathbb{R}[z_1, \cdots, z_{2n}]$ and $\textbf{a}$ be above the roots of $p$. Let $\delta_1, \cdots, \delta_{2n}$ be positive reals such that for each $i \in [n]$, 
\[\delta_i + \delta_{ n+i} \leq \dfrac{p}{\partial_i p+ \partial_{n+i} p}(\textbf{a}).\]
Then, letting, 
\[ q = \left[\prod_{i = 1}^{n} (\partial_i + \partial_{n+i})\right]p,\]
we have that for each $j \in [2n]$,
\[\Phi^{k}_{q}(\textbf{a}
 - \sum_{i = 1}^{2n} \delta_i e_i)   \leq \Phi_p^k(\textbf{a}).\]
In particular, $\textbf{a} - \sum_{i = 1}^{2n} \delta_i e_i$ is above the roots of $q$.

\end{lemma}

Let us now prove the Analytical Lieb-Sokal lemma, theorem(\ref{ALS}).

\begin{proof}
Let us apply lemma (\ref{shift2}) to the multiaffine real stable polynomial, 
\[ r(z_1+y_1, \cdots, z_n +y_n) = f(z_1, \cdots, z_n, y_1, \cdots, y_n) :=  \left[\prod_{i = 1}^{n} \left(\dfrac{\partial}{\partial z_i} + \dfrac{\partial}{\partial y_i}\right) \right]  p(\textbf{z})\overline{q}(\textbf{y}),\]
at a point above its roots $(\textbf{a}, \textbf{b})$. Letting $g(\textbf{z},\textbf{y}) =  p(\textbf{z})\overline{q}(\textbf{y})$, we note that, 
\[\dfrac{g}{\partial_{z_i} g+ \partial_{y_i} g}(\textbf{a}, \textbf{b}) = \dfrac{1}{\dfrac{\partial p}{\partial z_i} (\textbf{a})+ \dfrac{\partial \overline{q}}{\partial y_i}(\textbf{b})}.\]
We conclude that, 
 \[\textbf{c} := \textbf{a}+\textbf{b} - \left(\dfrac{1}{\varphi_1},\cdots, \dfrac{1}{\varphi_n}\right),\]
lies above the roots of $r = q(\partial)p$ where, 
\[\varphi_i = \dfrac{\partial_i \,p}{p}(\textbf{a}) + \dfrac{\partial_i \, \overline{q}}{\overline{q}}(\textbf{b}), \qquad i \in [n].\] 

\end{proof}

This result can easily be extended to the case when $p$ is not necessarily multiaffine, but $q$ still is. Recall the notion of the polarization of a polynomial $p(z_1, \cdots, z_n)$ of degree at most $k$ in any of the variables. Writing out, 
\[p(z_1, \cdots, z_n) = \sum_{\pmb{\kappa} = (\kappa_1, \cdots, \kappa_n) \in \mathbb{Z}_{+, \leq k}^{n}} a_{\pmb{\kappa}}z^{\pmb{\kappa}} = \sum_{\pmb{\kappa} = (\kappa_1, \cdots, \kappa_n) \in \mathbb{Z}_{+,\leq k}^{n}} a_{\pmb{\kappa}}\prod_{i =1}^{n}z_i^{\kappa_i},\]
one defines,
\[Pol(p)(z_{1,1}, \cdots, z_{1,k},z_{2,1}, \cdots, z_{2,k}, \cdots, z_{n,1}, \cdots, z_{n,k})  := \sum_{\pmb{\kappa} = (\kappa_1, \cdots, \kappa_n) \in \mathbb{Z}_{+, \leq k}^{n}} a_{\pmb{\kappa}}\prod_{i =1}^{n}\dfrac{e_{\kappa_i}(z_{i,1}, \cdots, z_{i,k})}{\binom{k}{\kappa_i}}.\]

Note that $\operatorname{Pol}(p)$ is multiaffine. A well known result of Borcea and Branden, an elegant application of the Grace-Walsh-Szego theorem, shows that polarization preserves real stability. We will also use the expression ``$\operatorname{Sym}$'' to denote the reverse operation.  One has the following simple lemma.
\begin{lemma}\label{noma}
Let $p(z_1, \cdots, z_n)$ be a polynomial of maximal degree at most $k$ in each of its variables and let $q(z_1, \cdots, z_n)$ be a multiaffine polynomial. Define, 
\[\tilde{q}(z_{1,1}, \cdots, z_{1,k},z_{2,1}, \cdots, z_{2,k}, \cdots, z_{n,1}, \cdots, z_{n,k}):= q(z_{1,1}+ \cdots + z_{1,k},z_{2,1}+\cdots+ z_{2,k}, \cdots, z_{n,1} +\cdots + z_{n,k}).\]
Then,
\[\operatorname{Sym}\left(\tilde{q}(\partial)\operatorname{Pol}(p)\right) = q(\partial)p.\]

\end{lemma}
\begin{proof}
By bilinearity, it is enough to prove this in the case $p$ and $q$ are monomials. Further, one readily checks that it is enough to prove this in the case when $q = z_1$ and $p = z_1^{m}$. One then has that $q(\partial)p = mz_{1}^{m-1}$, while, 
\begin{eqnarray*}
\operatorname{Sym}\left(\tilde{q}(\partial)Pol(p) \right)&=& \operatorname{Sym}\left((\sum_{i = 1}^{k}\partial_{1, i})\dfrac{e_{m}(z_{1,1}, \cdots, z_{1,k})}{\binom{k}{m}}\right)\\
&=& k\operatorname{Sym} \left(\dfrac{e_{m-1}(z_{1,2}, \cdots, z_{1,k})}{\binom{k}{m}}\right)\\
&=& k\,z_1^{m-1} \dfrac{\binom{k-1}{m-1}}{\binom{k}{m}}\\
&=& mz_1^{m-1}.
\end{eqnarray*}
This concludes the proof.
\end{proof}

Combining lemma (\ref{noma}) with theorem (\ref{ALS}) yields the following result.

	\begin{theorem}\label{ALS2}
Let $p$ and $q$ be real stable polynomials in $\mathbb{R}[z_1, \cdots, z_n]$ and let $q$ be additionally, multiaffine and let $r = q(\partial)p$. Let $\textbf{a}$ lie above the roots of $p$ and $\textbf{b}$ lie above the roots of $\overline{q}$. Then, 
\[\textbf{c} := \textbf{a}+\textbf{b} - \left(\dfrac{1}{\varphi_1},\cdots, \dfrac{1}{\varphi_n}\right),\]
lies above the roots of $q(\partial)p$ where, 
\[\varphi_i = \dfrac{\partial_{i,1} \,\operatorname{Pol}(p)}{\operatorname{Pol}(p)}(\overline{\textbf{a}}) + \dfrac{\partial_{i,1} \, \tilde{\overline{q}}}{\tilde{\overline{q}}}(\overline{\textbf{b}}), \qquad i \in [n].\]
Here, $\overline{\textbf{a}} = (\overbrace{a_1, \cdots, a_1}^{k}, \cdots, \overbrace{a_n, \cdots, a_n}^{k})$ if $\textbf{a} = (a_1, \cdots, a_n)$ and $\overline{\textbf{b}}$ is defined similarly. 
\end{theorem}

We remark here that for multiaffine polynomimals, we have that,
\[\dfrac{\partial_{i,1} \, \tilde{q}}{\tilde{q}}(\overline{\textbf{b}}) = \dfrac{\partial_{i} \, q}{q}(\textbf{b}).\]
Together with the fact that $\operatorname{Pol}(p) = p$ for multiaffine polynomials, this shows that theorem (\ref{ALS2}) reduces to theorem (\ref{ALS}) when both $p$ and $q$ are multiaffine. 

One of our motivations in proving this result was attempting to get optimal paving estimates in the Kadison-Singer problem. Though this does not yield the desired estimates, it does give non-trivial estimates. There are two ways expected characteristic polynomials can be used to get paving estimates. One could use \emph{mixed discriminants} or one could use \emph{mixed determinants}.

\section{Analytical Lieb-Sokal  and Strongly Rayleigh measures}

Fix a Strongly Rayleigh measure $\mu$ on $2^{[n]}$ and let $A \in M_n(\mathbb{C})$ be hermitian. Recall that this means that the generating polynomial $P_{\mu}$ of the measure, 
\[P_{mu}(z_1, \cdots, z_n) = \sum_{S \subset [n]}\mu(S)z^{S},\]
is real stable. The interlacing families method of MSS can be applied in this setting to relate the eigenvalues of the principal submatrices gotten by sampling with respect to $\mu$ and their expected characteristic polynomial. The theorem that follows was  proved by the author in \cite{MRKSP}[Theorem 2.5]. This expression differs slightly from the one in that paper since we sample $\chi[A(S)]$ rather than $\chi[A_S]$ which leads to the appearance of the flip $\tilde{P}_{\mu}$ in place of $P_{\mu}$. 
  \begin{theorem}\label{SR}
 Let $\mu$ be a Strongly Rayleigh distribution on $\mathcal{P}([n])$ and let $A \in M_n(\mathbb{C})$ be hermitian. Then, 
 \[\mathbb{E} \chi[A(S)] = \sum_{S \subset [n]} \mu(S) \chi[A(S)], \]
 is real rooted and further, 
  \[\mathbb{P} \left[ \lambda_{max} \chi[A(S)] \leq \lambda_{max} \mathbb{E} \chi[A(S)] \right]  > 0.\] 
  Further, we have the following formula for the expected characteristic polynomial, 
   \[\mathbb{E} \chi[A(S)] = \tilde{P}_{\mu}(\partial_1, \cdots, \partial_n)\operatorname{det}[Z-A] \mid_{Z = xI}.\]
  \end{theorem}

The analytical Lieb-Sokal lemma, theorem (\ref{ALS}) can now be immediately applied. Before this, we recall some standard definitions. Given a probability measure $\mu$ on $2^{[n]}$, the marginal probability of an element $i \in [n]$ is the probability that $i$ belongs to a random subset of $[n]$. This marginal probability, denoted $\mathbb{P}_{i \sim S}[\mu]$, can also be expressed as, 
\[\mathbb{P}_{i \sim S}[\mu] = \partial_i P_{\mu}\mid_{z_1 = \cdots = z_n =  1}.\]
Theorem (\ref{ALS}) yields the following quantitative estimate, 
\begin{theorem}\label{SRALS}
 Let $\mu$ be a homogeneous Strongly Rayleigh distribution on $\mathcal{P}([n])$ such that the marginal probabilities of all elements are at most $\epsilon$.  Let $A \in M_n(\mathbb{C})$ be a PSD contraction such that its diagonal entries are all at most $\alpha$. Then, provided that $\sqrt{\alpha} + \sqrt{\epsilon} \leq 1$, we have that, 
 \[\mathbb{P}\left[\lambda_{max}\,A(S) < 4\epsilon^{1/4}+\alpha\right]> 0.\]
\end{theorem}
\begin{proof}
Since $P_{\mu}$ has positive coefficients and $A$ is a contraction, we have that $a$ is above the roots of $p(Z) = P_{\mu}$ for any $a > 0$ and $b$ is above the roots of $q(Z) = \operatorname{det}[Z-A]$ for any $b > 1$. Using theorems (\ref{SR}) and (\ref{ALS}), we see that,
\[f(a, b) = a + b - \dfrac{1}{\varphi_1 (a \bold{1})+ \varphi_2 (b\bold{1})},\]
where, 
\[\varphi_1(a \bold{1}) = \operatorname{max}_{k \in [n]} \dfrac{\partial_i p}{p}(a \bold{1}) , \quad \varphi_2(b \bold{1}) = \operatorname{max}_{k \in [n]} \dfrac{\partial_i q}{q}(b \bold{1}),\]
is above the roots of $\mathbb{E} \chi[A(S)]$ for any $a > 0$ and $b > 1$. The hypothesis on the marginal probabilities of $\mu$, together with the fact that $\mu$ is homogeneous, shows that $\varphi_1(a \bold{1}) \leq \epsilon/a$. A routine calculation, see for instance \cite{MRKSP}, shows that, 
\[\varphi_2(b \bold{1})   \leq   \dfrac{\alpha}{b-1} +\dfrac{1-\alpha}{b}.\]
We now see that, $c\bold{1}$ is above the roots of $\mathbb{E} \chi[A(S)]$, where
\[c = f(a, b) := a + b - \dfrac{1}{\dfrac{\epsilon}{a} +  \dfrac{\alpha}{b-1} +\dfrac{1-\alpha}{b}},\quad \text{ for any } a > 0, b > 1,\]
The function $a \longrightarrow a - (\epsilon/a+x)^{-1}$, where $\epsilon, x > 0$, has minimum value $(1-\sqrt{\epsilon})^2/x$. We see that, 
\begin{eqnarray*}
\operatorname{min}_{\substack{a > 0,\\ b> 1}} f(a, b) &=& \operatorname{min}_{b > 1}b - \left(1-\sqrt{\epsilon}\right)^2 \dfrac{b(b-1)}{b-1+\alpha},\\
&=& \operatorname{min}_{b > 1}b -  \left(1-\sqrt{\epsilon}\right)^2  \left[(b-1+\alpha) + 1-2\alpha - \dfrac{\alpha(1-\alpha)}{b-1+\alpha}\right],\\
&=& \operatorname{min}_{y >\alpha} y (2\sqrt{\epsilon} - \epsilon) + \dfrac{\alpha(1-\alpha)}{y}\left(1-\sqrt{\epsilon}\right)^2 + 1-\alpha - (1-2\alpha)\left(1-\sqrt{\epsilon}\right)^2,  
\end{eqnarray*}
where $y = b-1+\alpha$. This last expression has minimizer,
\[\gamma:= 2\sqrt{2\sqrt{\epsilon} - \epsilon}\left(1-\sqrt{\epsilon}\right)\sqrt{\alpha(1-\alpha)} + \alpha + (1-2\alpha)(2\sqrt{\epsilon}-\epsilon) ,\]
if $\sqrt{\alpha} + \sqrt{\epsilon} \leq 1$ and with the minimizer being the trivial estimate $1$ otherwise. Noting that $\alpha(1-\alpha) \leq 1/4$, that $1-2\alpha \leq 1$ and that $\epsilon \leq 1$, we see that, 
\[\gamma \leq \sqrt{2}\epsilon^{1/4} + 2\sqrt{\epsilon} + \alpha \leq 4 \epsilon^{1/4} + \alpha.\]
\end{proof}

We now record an application. 
\begin{proposition}\label{mdpaving}
Let $A \in M_n(\mathbb{C})$ be a PSD contraction such that its diagonal entries are all at most $\alpha$. Then, there there is a $r$ paving of $A$ whose largest eigenvalue is at most $4 r^{-1/4} + \alpha$.
\end{proposition}
\begin{proof}
Let $B = \overbrace{A \oplus \cdots \oplus A}^{r}$ and let $\mu$ be the uniform measure supported on 
\[X = \{(S_1, \cdots, S_r) \mid [n]^{r}, \, S_1 \amalg \cdots \amalg S_r = [n]\}.\]
Associating the number $i$ in the $j$'th copy of $[n]$ with the variable $z_{i, j}$, we see that the generating polynomial of $\mu$ is given by, 
\[P_{\mu} = \prod_{i = 1}^{n} \left(z_{i, 1} + \cdots + z_{i, r}\right),\]
which is easily seen to be Strongly Rayleigh. The marginal probabilities are all $1/r$ and theorem (\ref{SRALS}) yields that there is a subset $S \in X$ such that, 
\[\lambda_{max}\chi[B(S)] \leq 4 r^{-1/4} + \alpha.\]
Finally, we notice that $B(S) = A(S_1) \oplus \cdots \oplus A(S_r)$, yielding an $r$ paving of $A$ with the claimed bound for its largest eigenvalue. 
\end{proof}

We would like to point out that a similar result was obtained by Nima Anari and Shayan Oveis Gharan \cite{AnaGha1}, in the context of estimating the norms of outer products of random vectors drawn from Strongly Rayeigh distributions. Their estimate is stronger than ours in certain regimes and weaker in others.

\subsection{Equal sized pavings}

With the Analytical Lieb-Sokal lemma in hand, one could use other other strongly Rayleigh measures to prove paving estimates. We prove here a result concerning pavings with size restrictions. Given an isotropic collection of vectors in $\mathbb{C}^n$, can one get a $r$ paving where all the subsets have the same cardinality? We will use the analytical Lieb-Sokal lemma to show that this is indeed the case. Note that is the same problem as finding a size $r$ paving of a hermitian matrix where all the blocks have the same size. 

Fix an $r \in \mathbb{N}$ and let $A \in \mathbb{C}^{rm}$ be a hermitian zero diagonal contraction. Identify the sets, $[r^2m] \cong \overbrace{[rm] \amalg \cdots \amalg [rm]}^{r}$ and associate the variable $z_{i, j}$ with $i$ in the $j$'th copy of $[rm]$ in $\overbrace{[rm] \amalg \cdots \amalg [rm]}^{r}$. Consider the uniform measure $\mu$ on  
\[\{(S_1,\cdots,  S_r) \subset \overbrace{[rm] \times \cdots \times [rm]}^r \mid S_1 \amalg \cdots \amalg S_r = [rm],\, |S_1| = \cdots = |S_r| = m\}.\]
We claim that this measure is Strongly Rayleigh. This can be seen by the fact that the generating polynomial of this measure is the Hadamard product of the uniform measures $\mu_1$ and $\mu_2$, supported on
\[\{(S_1,\cdots,  S_r) \subset \overbrace{[rm] \times \cdots \times [rm]}^r \mid S_1 \amalg \cdots \amalg S_r = [rm]\}.\]
and
\[\{(S_1,\cdots,  S_r) \subset \overbrace{[rm] \times \cdots \times [rm]}^r \mid  |S_1| = \cdots = |S_r| = m\}.\]

Both of these measures are strongly Rayleigh. This is because their generating polynomials are 
\[\prod_{i = 1}^{rm}(z_{i,1} + \cdots + z_{i,r}),\]
 and 
\[\prod_{j = 1}^{r} \left(\sum_{S \subset [rm]} \prod_{i \in S} z_{i,j}\right) = \prod_{j = 1}^{r} e_{m}(z_{1,j}, \cdots, z_{rm,j}),\] respectively. A theorem of Borcea and Branden shows that Hadamard products of multiaffine real stable polynomials are real stable. The marginal probabilities of the measure $\mu$ are all easily seen to be equal to $1/r$. Identical to the proof of proposition (\ref{mdpaving}), we have,
\begin{proposition}\label{mdpavingequal}
Let $A \in M_{rn}(\mathbb{C})$ be a PSD contraction such that its diagonal entries are all at most $\alpha$. Then, there there is a $r$ paving of $A$ whose blocks all have size equal to $n$ and whose largest eigenvalue is at most $4 r^{-1/4} + \alpha$.
\end{proposition}

\subsection{Mixed discriminants}
Marcus, Spielman and Srivastava \cite{MSS2}, settled the Kadison-Singer problem by solving Weaver's $KS_r$ conjecture and in their proof proved bounds on the largest root of \emph{mixed characteristic polynomials}, that are closely related to \emph{mixed discriminants}.  Recall that given PSD matrices $A_1, \cdots, A_m \in M_n(\mathbb{C})$, such that $A_1 + \cdots + A_m = I_n$, the mixed characteristic polynomial $\mu[A_1, \cdots, A_m]$ is defined as $p(x\textbf{1}_m)$ where, 
\[p(z_1, \cdots, z_m) :=\left[ \prod_{i = 1}^{m} \left(1-\dfrac{\partial}{\partial z_i}\right)\right] \operatorname{det}\left[ z_1 A_1 + \cdots + z_m A_m\right].\]

MSS proved the followıng fundamental estimate on the roots of mixed characteristic polynomials. 
\begin{theorem}[MSS]
Given PSD matrices $A_1, \cdots, A_m \in M_n(\mathbb{C})$, such that $A_1 + \cdots + A_m = I_n$ and with $\operatorname{Trace}[A_k] \leq \epsilon$ for $k \in [m]$ and letting $p$ be as above, one has that $(1+\sqrt{\epsilon})^2\textbf{1}$ is above the roots of $p$. 
\end{theorem}

The mixed characteristic polynomials that arise when working with the Kadison-Singer problem have an additional special structure. For this class, the root estimates of MSS take the following form.
\begin{theorem}[MSS]
Let $A_1, \cdots, A_m$ be rank $1$ $ n\times n$ PSD matrices such that, 
\[A_1 + \cdots + A_m = I_n, \quad \operatorname{Trace}(A_k) \leq \epsilon, \quad k \in [m],\]
and let $r \in \mathbb{N}$. Then, letting, 
\[p(z_1, \cdots, z_m) :=\left[ \prod_{i = 1}^{m} \left(1-\dfrac{\partial}{\partial z_i}\right)\right] \operatorname{det}\left[ z_1 A_1 + \cdots + z_m A_m\right]^r \mid_{z_1 = \cdots = z_m = x},\]
we have that $(1+\sqrt{r\epsilon})^2 \textbf{1}$ is above the roots of $p$.
\end{theorem}

This implies that for $\epsilon < \left(1-1/\sqrt{r}\right)^2$, we have that $r\textbf{1}$ is above the roots of $p$. This is precisely what is needed to complete the proof of Kadison-Singer.

We now show how the Analytical Lieb-Sokal lemma can also be used to get non-trivial estimates. We will use theorem (\ref{ALS2}) for this. We note that the polarization of the polynomial $p(z) = \operatorname{det}\left[ z_1 A_1 + \cdots + z_m A_m\right]^r$ is given by, 
\[Pol(p) = \prod_{j = 1}^{r} \operatorname{det}\left[ z_{1,j} A_1 + \cdots + z_{m,j} A_m\right].\]
We let, 
\[q(Z) = \prod_{i = 1}^{m} \left(1-z_i \right) .\]
In the notation of theorem (\ref{ALS2}), we have that,
\[\tilde{q} =  \prod_{i = 1}^{m} \left(1- \sum_{j = 1}^{r} z_{i,j} \right).\]
The \emph{flip} $\overline{\tilde{q}}$ of $\tilde{q}$ is easily calculated. One has, 
\[\overline{\tilde{q}} = \prod_{i = 1}^{m} \left(\prod_{j = 1}^{r} z_{i,j} - \sum_{k = 1}^{r} \prod_{ j \neq k}z_{i, j} \right)  .\]
One sees that $a$ is above the roots of $p$ for $a > 0$ and that $b$ is above the roots of $\overline{\tilde{q}}$ for $b > r$. One further, calculates that,
\[\dfrac{\partial_{i,j}Pol( p)}{Pol(p)}(a\bold{1}) \leq \dfrac{\epsilon}{a}, \qquad \dfrac{\partial_{i,j} \overline{q}}{q}(b\bold{1}) = \dfrac{b^{r-2}(b-r+1)}{b^{r-1}(b - r)} = \dfrac{b-r+1}{b(b - r)}.\]
 Here, we assume that $\operatorname{Trace}(A_i) \leq \epsilon$ for all $i \in [m]$. Theorem (\ref{ALS2}) now shows that, $c\bold{1}$ is above the roots of $q(\partial)p$ where, 
\[ c = f(a, b) := a + b - \dfrac{1}{\dfrac{\epsilon}{a} + \dfrac{b-r+1}{b(b-r)}}.\]
Optimizing over $a$, we see that, 
\[\operatorname{min}_{a > 0}f(a, b) = b-(1-\sqrt{\epsilon})^2 \dfrac{b(b-r)}{b-r+1}.\]
Further optimizing over $b$, we see that if $\epsilon \leq (1 - \sqrt{1/r})^2$, then $c\bold{1}$ is above the roots of $q(\partial p)$, where, 
\[c = \left(1+\sqrt{2r}\epsilon^{1/4}\right)^2.\]

This must be compared to the result of MSS \cite{MSS2}, where they get the asymptotically stronger estimate, $(1+\sqrt{r\epsilon})^2$. 
\section{Concluding remarks}

A major drawback of the main theorem (\ref{ALS}) is the suboptimal estimates ($O(r^{-1/4})$ in place of the optimal $O(r^{-1/2})$) when one applies it to natural determinantal polynomials. It would be intersting to see if working with other potential/barrier functions can fix this problem. It is also interesting to see if one can get (one sided) pavings into equal sized, as in proposition (\ref{mdpavingequal}) blocks of size $O(1/\epsilon^2)$ using other, more adhoc, techniques.

	
	   \end{document}